\documentclass[12pt]{article}

\usepackage[T1]{fontenc}
\usepackage[active]{srcltx}
\usepackage{lmodern}
\usepackage{amsmath,amstext,amssymb,graphicx,graphics,color}
\usepackage{tikz}
\usepackage{mleftright}
\usepackage{pgfplots}
\usepackage{theorem}
\usepackage{cite}               
\usepackage{hyperref}           
\usepackage{bbm}                
\usepackage{fancybox}           
\usepackage[section]{placeins}  
\usepackage{subcaption}
\theorembodyfont{\normalfont\slshape} 
\newtheorem{definition}{Definition}
\newtheorem{theorem}{Theorem}
\newtheorem{proposition}{Proposition}[section]

\theoremstyle{break} 

\newenvironment{proof}%
{{\par\noindent \bf Proof. \nobreak}}%
{\nobreak \removelastskip \nobreak \hfill $\Box$ \medbreak}
{{\par\noindent \bf Proof \nobreak}}%
{\nobreak \removelastskip \nobreak \hfill $\Box$ \medbreak}
{{\par\noindent \bf Proof lemma. \nobreak}}%
{\nobreak \removelastskip \nobreak \bf End proof lemma. \medbreak}
\newenvironment{remark}{\par \medskip \noindent {\bf Remark. }\nobreak}{\par \medskip}

\usepackage{fancyhdr}
 \fancypagestyle{plain}{
  \fancyhead{}
  
}

\def\paragraph#1{{\bf #1\ }}
\newcommand{\RN}[1]{%
  \textup{\uppercase\expandafter{\romannumeral#1}}%
}
\newcommand{\expo}{\mathrm{e}}

\newcommand{\dd}{\mathrm{d}}

\newcommand{\DD}{\mathrm{D}}
\newcommand{\KL}{\mathrm{KL}}
\newcommand{\NN}{\mathrm{N}}

\newcommand{\E}{\mathrm{E}}
\newcommand{\VV}{\mathrm{V}}
\newcommand{\overbar}[1]{\mkern 1.5mu\overline{\mkern-1.5mu#1\mkern-1.5mu}\mkern 1.5mu}
\usepackage[utf8]{inputenc}
\usepackage{unicode_character}

\title{Uncovering a two-phase dynamics from a dollar exchange model with bank and debt}

\author{Fei Cao \footnotemark[1] \and S\'ebastien Motsch \footnotemark[2]}

\begin{document}
\maketitle

\footnotetext[1]{University of Massachusetts Amherst - Department of Mathematics and Statistics, Amherst, MA 01003, USA}
\footnotetext[2]{Arizona State University - School of Mathematical and Statistical Sciences, 900 S Palm Walk, Tempe, AZ 85287-1804, USA}

\tableofcontents

\begin{abstract}
We investigate the unbiased model for money exchanges with collective debt limit: agents give at random time a dollar to one another as long as  they have at least one dollar or they can borrow a dollar from a central bank if the bank is not empty. Surprisingly, this dynamics eventually leads to an asymmetric Laplace distribution of wealth (conjectured in \cite{xi_required_2005} and shown formally in a recent work \cite{lanchier_rigorous_2018-1}). In this manuscript, we carry out a formal mean-field limit as the number of agents goes to infinity where we uncover a two-phase (ODE) dynamics. Convergence towards the unique equilibrium (two-sided geometric) distribution in the large time limit is also shown and the role played by the bank and debt (in terms of Gini index or wealth inequality) will be explored numerically as well.
\end{abstract}

\noindent {\bf Key words: Econophysics, Agent-based model, Mean-field, Two-phase, Bank}

\section{Introduction}
\setcounter{equation}{0}

Econophysics is a subfield of statistical physics that apply concepts and techniques from traditional physics to economics or finance \cite{dragulescu_statistical_2000}. One of the primary goal of this area of research is to explain how various economical phenomena could be derived from universal laws in statistical physics under certain model assumptions, and we refer to \cite{kutner_econophysics_2019} for a general review.

There several motivations for the study of econophysics models: from the perspective of a policy maker, how to influence/control the emerging wealth inequality (measured by Gini index) in order to mitigate the alarming gap between rich and poor is a central issue to be dealt with. From a mathematical viewpoint, the fundamental mechanisms behind the formation of macroscopic phenomena, for instance various possible wealth distributions from different agent-based money exchange models, must be thoroughly understood. For a given (stochastic) agent-based model, we would like to identify a limit (deterministic) dynamics when we send the number of individuals to infinity, and then the deterministic system will be further analyzed with the intention of proving its convergence to equilibrium (if there is one) for large time. This paradigm has been implemented in vast amount of works across different fields of applied mathematics, see for instance \cite{bresch_mean_2019,cao_k_2021,carlen_kinetic_2013,motsch_short_2018}.

In this work, we consider a simple mechanism for money exchange involving a bank, meaning that there are a fixed number of agents (denoted by $N$) and one bank. We denote by $S_i(t)$ the amount of dollars the agent $i$ has at time $t$ and we suppose that $\sum_{i=1}^N S_i(0) = N\,\mu$ for some fixed $\mu \in \mathbb{R}_+$ so that $N\,\mu \in \mathbb{N}_+$. Moreover, we denote by $B_* := B_c + B_d$ the initial amount of dollars in the bank, where $B_c$ and $B_d$ represent the amount of dollars owned by the bank in the form of ``cash'' and in the form of ``debt'' (borrowed by agents), respectively. Also, we introduce another parameter $\nu \in \mathbb{R}_+$ so that $N\,\mu\,\nu \in \mathbb{N}_+$ and set $B_* = N\,\mu\,\nu$.

The model investigated in this work was proposed in \cite{xi_required_2005} and revisited by \cite{lanchier_rigorous_2018-1}: at random time (generated by an exponential law), an agent $i$ (the ``giver'') and an agent $j$ (the ``receiver'') are picked uniformly at random. If the ``giver'' $i$ has at least one dollar (i.e. $S_i\geq 1$) or if the central bank has ``cash'' (i.e. $B_c\geq 1$), then the receiver $j$ receives a dollar. Otherwise, when the receiver $i$ has no dollar and the bank has no cash, then nothing happens. We illustrate the dynamics in figure \ref{fig:illustration_model} explaining the three cases when agent $i$ is picked to give one dollar to agent $j$. From now on, we will call this model the \textbf{unbiased exchange model with collective debt limit} and it can be represented by \eqref{unbiased_exchange_with_debt}.
\begin{equation}
\label{unbiased_exchange_with_debt}
(S_i,S_j)~ \begin{tikzpicture} \draw [->,decorate,decoration={snake,amplitude=.4mm,segment length=2mm,post length=1mm}]
  (0,0) -- (.6,0); \node[above,red] at (0.3,0) {\small{$λ$}};\end{tikzpicture}~  (S_i-1,S_j+1) \qquad \text{ if}~ S_i\geq 1~ \text{{\bf or}}~ B_c \geq 1.
\end{equation}
Notice that when the bank gives a dollar to agent $j$, there is still one dollar withdrew from the giver $i$, i.e. the debt of agent $i$ increases (represented in red in figure \ref{fig:illustration_model}). The debt of agent $i$ could be reduced once it will become a ``receiver''. It is also important to notice that in this model the bank never losses money, it just transforms its ``cash'' $B_c$ into ``debt''.  Without the bank, agents can only give a dollar when they have at least one dollar (hence no debt is allowed). In this case, the model is termed as the one-coin model in \cite{lanchier_rigorous_2017}, the unbiased exchange model in \cite{cao_derivation_2021,cao_interacting_2022}, and the mean-field zero range process in \cite{merle_cutoff_2019}.

\begin{remark}
In order to have the correct asymptotic as the number of agents goes to infinity $N→+∞$, we need to adjust the rate $λ$  by normalizing by $N$ so that the rate of a typical agent giving a dollar per unit time is of order $1$.
\end{remark}

\begin{figure}[p]
\centering
\includegraphics[width=.97\textwidth]{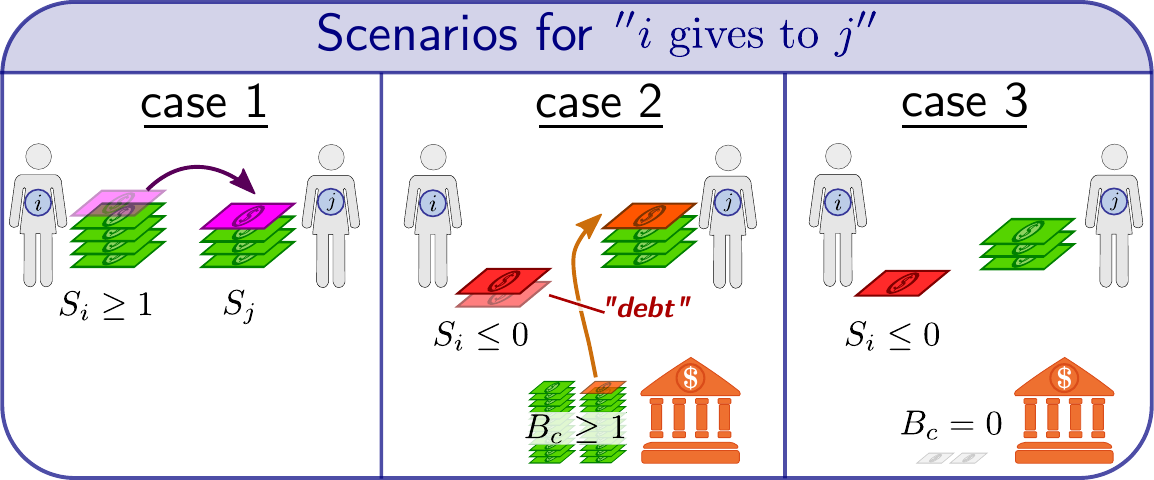}
\caption{Illustration of the unbiased exchange model with collective debt limit: at random time, a ``giver'' $i$ has to give one dollar to a ``receiver'' $j$. Three cases are possible. {\bf Case 1}: if the ``giver'' $i$ has at least a dollar (i.e. $S_i≥1$), then it gives it to $j$. {\bf Case 2}: if $i$ does not have a dollar ($S_i≤0$) and the central bank has cash (i.e. $B_c≥1$), then $j$ receives one dollar from the bank and the debt of $i$ is increased by one. {\bf Case 3}: the giver $i$ and the bank do not have any money, nothing happens.}
\label{fig:illustration_model}
\end{figure}

\begin{figure}[p]
\centering
\includegraphics[width=.97\textwidth]{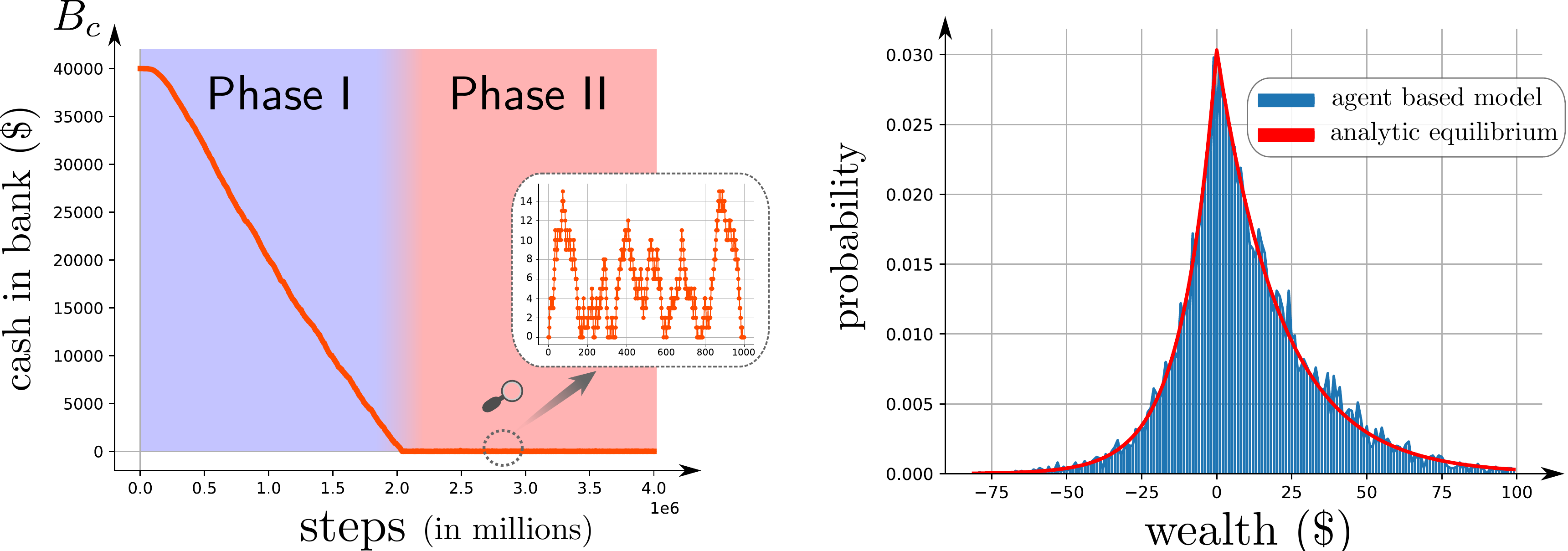}
\caption{{\bf Left:} The evolution of the amount of cash in the bank $B_c$ when running a simulation with $N=10,000$ agents with $μ=10$ dollars per agent. In a first phase, $B_c$ decays linearly until it reaches zero. After that, in the second phase $B_c$ remains close to zero. {\bf Right:} The distribution of wealth for the agent-based dynamics after $10,000,000$ steps. Notice that this distribution is well-approximated by an asymmetric Laplace distribution given by \eqref{eq:Laplace}-\eqref{eq:parameters} with $μ=10$ and $ν=.4$.}
\label{fig:numerical_simulation}
\end{figure}

The fundamental question of interest is the exploration of the limiting money distribution among the agents as the total number of agents $N$ and time $t$ become large. To foresee the behavior of the dynamics under these limits, we perform a numerical simulation with  $N=10,000$ agents over $10,000,000$ steps. In figure \ref{fig:numerical_simulation}-left, we plot the evolution of the cash in the bank $B_c$ over time. We observe a first phase where the cash $B_c$ is decaying linearly until it reaches zero (after roughly $2$ millions steps). After that, there is a second phase where $B_c$ remains close to zero. In figure \ref{fig:numerical_simulation}-right, the distribution of the wealth distribution is plotted after $10$ millions steps. This distribution is well approximated by an asymmetric Laplace distribution given in \eqref{eq:Laplace}.

This numerical result will be explained by our derivation in section \ref{sec:mean_field} and the following asymptotic analysis in section \ref{sec:large_time_limit}. Our approach illustrated by figure \ref{fig:scheme_sketch} consists in two steps. Our first step consists in deriving the limit dynamics as the number of agents goes to infinity $N→+∞$ (section \ref{sec:mean_field}). With this aim, we introduce the probability distribution of wealth:
\begin{equation}
  \label{eq:p}
  {\bf p}(t)=\left(\ldots,p_{-n}(t),\ldots,p_{-1}(t),p_0(t),p_1(t),\ldots,p_n(t),\ldots\right) 
\end{equation}
with $p_n(t)= \{``\text{probability that a typical agent has } n \text{ dollars at time}~ t "\}$. The evolution of ${\bf p}(t)$ will be given by a (deterministic) nonlinear ordinary differential equations. To fully justify this transition from a stochastic interacting systems into a deterministic set of ODEs, one needs the so-called \emph{propagation of chaos} \cite{sznitman_topics_1991}. We do not investigate the proof in this manuscript, the derivation has been rigorously justified in various models arising from econophysics, see for instance \cite{cao_derivation_2021,cao_entropy_2021,cao_explicit_2021,cao_interacting_2022,graham_rate_2009,cortez_quantitative_2016,cortez_particle_2018}. The additional difficulty here is that the evolution of ${\bf p}(t)$ is split into two phases. Indeed, the evolution changes at the first time when there is no more cash in the bank. We will denote by $t_*$ the time at which such event occur, i.e. at $t_*$ the bank is empty for the first time. The evolution equation of ${\bf p}(t)$ takes the form of:
\begin{equation}
  \label{eq:evolution_p}
  \begin{array}{cccc}
  \text{\bf Phase I:} &  \hspace{.5cm}  ∂_t {\bf p} = λ Q_1[{\bf p}] \hspace{.5cm} & \text{for} & 0≤t≤t_* \\ \smallskip
  \text{\bf Phase II:} & \hspace{.5cm} ∂_t {\bf p} = λ Q_2[{\bf p}] \hspace{.5cm} & \text{for} & t>t_*
  \end{array}
\end{equation}
where the exact expressions for the operator $Q_1$ and $Q_2$ are given by \eqref{eq:Q_1_b} and \eqref{eq:Q_2_b}, respectively.

\begin{figure}[!htb]
  \centering
  \includegraphics[width=.97\textwidth]{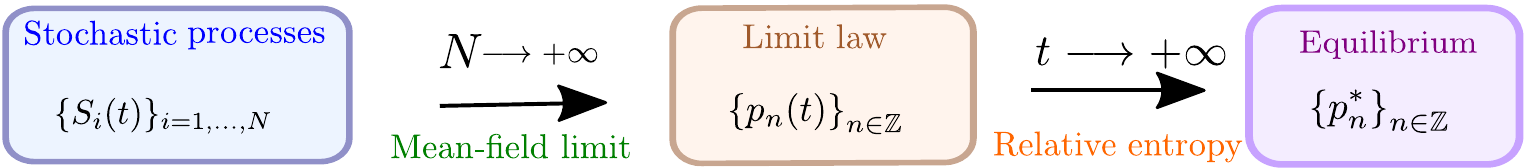}
\caption{Schematic illustration of the strategy of the limiting procedures.}
\label{fig:scheme_sketch}
\end{figure}

Our second step is to investigate the asymptotic behavior of the probability mass function ${\bf p}(t)$ as $t→+∞$, more precisely its convergence toward an asymmetric Laplace distribution given in \eqref{eq:Laplace}:
\begin{equation}\label{eq:Laplace}
\rho(x) = \begin{cases}
\rho_0\,\expo^{-\alpha\,x},&\quad \text{for}~ x \geq 0,\\
\rho_0\,\expo^{\beta\,x},&\quad \text{for}~ x \leq 0,
\end{cases}
\end{equation}
where the three positive parameters $\rho_0$, $\alpha$ and $\beta$ only depend on $μ$ (average wealth per agent) and $ν$ (the ratio between the bank and agent's combined wealth). Our proof relies on the so-called entropy method which gives a rigorous proof of the convergence of ${\bf p}(t)$ toward the asymmetric Laplace distribution $\rho$, see section \ref{subsec:3.2}. Moreover, a standard linearization analysis performed in section \ref{subsec:3.3} allows us to obtain an exponential decay result for the linearized entropy (near equilibrium), provided that the parameters $\mu$ and $\nu$ fulfill certain criteria.

Finally, we would like explore the role played by the bank in terms of wealth inequality. In particular, we conjecture that the inclusion of the bank and possibility for agents to go into debt will lead to accentuation of wealth inequality (measured by the Gini index), compared to the usual unbiased exchange model (without the presence of a bank). This type of Gini index comparison conjecture (shown numerically in section \ref{sec:Gini_comparison}) makes sense at least heuristically but we are not able to provide a rigorous proof.

Although we will only investigate a specific binary exchange models in the present work, other exchange rules can also be imposed and studied, leading to different models. To name a few, the so-called immediate exchange model introduced in \cite{heinsalu_kinetic_2014} assumes that pairs of agents are randomly and uniformly picked at each random time, and each of the agents transfer a random fraction of its money to the other agents, where these fractions are independent and uniformly distributed in $[0,1]$. The so-called uniform reshuffling model investigated in \cite{dragulescu_statistical_2000,lanchier_rigorous_2018,cao_entropy_2021} suggests that the total amount of money of two randomly and uniformly picked agents possess before interaction is uniformly redistributed among the two agents after interaction. For closely related variants of the unbiased exchange model we refer to the recent work \cite{cao_derivation_2021}. For other models arising from econophysics, see \cite{chakraborti_statistical_2000, chatterjee_pareto_2004, lanchier_rigorous_2018-1,cao_explicit_2021}.

\section{Formal mean-field limit}\label{sec:mean_field}
\setcounter{equation}{0}

In this section, we carry out a formal mean-field argument of the stochastic agent-based dynamics \eqref{unbiased_exchange_with_debt} as the number of agents $N$ goes to infinity. Even our analysis is not completely rigorous, the resulting system of ODEs (i.e., \eqref{eq:Q_1_a}-\eqref{eq:Q_1_b} in \text{\bf Phase I} and \eqref{eq:Q_2_a}-\eqref{eq:Q_2_b} in \text{\bf Phase II}) admits a unique equilibrium distribution as $t \to \infty$, which is a two-sided geometric distribution and can be well-approximated by an asymmetric Laplace distribution when $\mu \gg 1$ (i.e., when the initial average amount of dollars per agent becomes large). As our predication of the limiting distribution of money based on the large time behavior of the ODE system \eqref{eq:Q_2_a}-\eqref{eq:Q_2_b} coincides with the conjecture in \cite{xi_required_2005} and the work in \cite{lanchier_rigorous_2018-1}, it strongly indicates that our heuristic mean-field analysis actually captures the correct deterministic behavior of the underlying stochastic dynamics when we send $N \to \infty$.

The unbiased exchange model with collective debt limit can be written in terms of a system of stochastic differential equations, thanks to the framework set up in \cite{cao_derivation_2021} for the study of the basic unbiased exchange model as well as some of its variants. Introducing $\{\mathrm{N}_t^{(i,j)}\}_{1\leq i,j\leq N}$, which represents a collection of independent Poisson processes with constant intensity $\frac{\lambda}{N}$, the evolution of each $S_i$ is given by:
\begin{equation}
  \label{eq:SDE}
  \dd S_i = -\sum \limits^{N}_{j=1} \underbrace{\left(1-\mathbbm{1}_{(-\infty,0]}(S_i)\cdot \delta_0(B_c)\right) \dd \mathrm{N}^{(i,j)}_{t}}_{\text{``$i$ gives to $j$''}} + \sum \limits^{N}_{j=1} \underbrace{\left(1-\mathbbm{1}_{(-\infty,0]}(S_j)\cdot \delta_0(B_c)\right) \dd \mathrm{N}^{(j,i)}_{t}}_{\text{``$j$ gives to $i$''}},
\end{equation}
where we use the notation $\delta_0(B_c) := \mathbbm{1}\{B_c = 0\}$. By the obvious symmetry, we can focus on the case when $i=1$ and notice that whenever $B_c \geq 1$, the SDE for $S_1$ simplifies to
\begin{equation}
\label{eq:SDE_phase1}
\dd S_1 = -\sum \limits^{N}_{j=1} \dd \NN^{(1,j)}_t + \sum \limits^{N}_{j=1} \dd \NN^{(j,1)}_t
\end{equation}
If we introduce
\begin{displaymath}
  \mathrm{\bf N}^1_t = \sum_{j=1}^N \mathrm{N}^{(1,j)}_t,\quad \mathrm{\bf M}^1_t = \sum_{j=1}^N \mathrm{N}^{(j,1)}_t,
\end{displaymath}
then the two Poisson processes $\mathrm{\bf N}^1_t$ and $\mathrm{\bf M}^1_t$ are of intensity $λ := 1$.  Motivated by \eqref{eq:SDE_phase1}, we give the following definition of the limiting dynamics of $S_1(t)$ as $N \rightarrow \infty$ from the process point of view, providing that $B_c \geq 1$.

\begin{definition}[\textbf{Mean-field equation --- Phase \RN{1}}]
We define $\overbar{S}_1$ to be the compound Poisson process satisfying the following SDE:
\begin{equation}
\label{eq:SDE_phase1_limit}
\dd \overbar{S}_1 = -\dd \overbar{\mathrm{\bf N}}^1_t + \dd \overbar{\mathrm{\bf M}}^1_t,
\end{equation}
in which $\overbar{\mathrm{\bf N}}^1_t$ and $\overbar{\mathrm{\bf M}}^1_t$ are independent Poisson processes with unit intensity.
\end{definition}

We denote by ${\bf p}(t)=\left(\ldots,p_{-n}(t),\ldots,p_{-1}(t),p_0(t),p_1(t),\ldots,p_n(t),\ldots\right)$ the law of the process $\overbar{S}_1(t)$, i.e. $p_n(t) = \mathbb P\left[\overbar{S}_1(t) = n\right]$. Its time evolution is given by the following dynamics:
\begin{equation}
    \label{eq:Q_1_a}
\frac{\dd}{\dd t} {\bf p}(t) = Q_1[{\bf p}(t)]
\end{equation}
with
\begin{equation}
  \label{eq:Q_1_b}
  Q_1[{\bf p}]_n := p_{n+1} + p_{n-1} - 2\,p_n \qquad \text{for } n ∈ℤ.
\end{equation}

\begin{remark}\label{remark1}
It is readily seen that ${\bf p}(t)$ is exactly the law of a continuous-time symmetric simple random walk on $\mathbb Z$ defined via
\begin{equation}\label{eq:X_t}
X_t := X_0 + \sum_{i=1}^{N_t} Y_i
\end{equation}
with $Y_i \in \{-1,1\}$ being a sequence of independent Rademacher random signs and $N_t$ being a Poisson clock running at unit rate.
\end{remark}

The distribution ${\bf p}(t)$ naturally preserves its mass (since it is a probability mass function) and the mean value (as the total amount of money in the whole system is preserved), thus if we introduce the affine subspaces:
\begin{equation}\label{eq:probability_space}
\mathcal{S}_\mu = \{{\bf p} \mid \sum_{n \in \mathbb Z} p_n =1,~p_n \geq 0,~\sum_{n \in \mathbb Z} n\,p_n =\mu\}
\end{equation}
and \[\mathcal{S}^+_\mu = \{{\bf p} \in \mathcal{S}_\mu \mid p_n = 0~\textrm{for}~ n < 0\},\] then it is clear that the unique solution ${\bf p}(t)$ of \eqref{eq:Q_1_a}-\eqref{eq:Q_1_b} with ${\bf p}(0) \in \mathcal{S}^+_\mu$ satisfies ${\bf p}(t) \in \mathcal{S}_\mu$ for all $t > 0$. Moreover, if we define the average amount of ``debt'' per agent as
\begin{equation}
  \label{eq:debt}
  D[{\bf p}] = -\sum_{n \leq -1} n\,p_n,
\end{equation}
this quantity will be non-decreasing: $\frac{\dd}{\dd t} D[{\bf p}(t)] = p_0(t) \geq 0$. Since the average amount of debt each agent can sustain in the underlying stochastic $N$-agents system is at most $\mu\,\nu$, we therefore terminate the evolution of \text{\bf Phase I} \eqref{eq:Q_1_a}-\eqref{eq:Q_1_b} until $t = t_*$, where
\begin{equation}
\label{eq:t_*}
t_* = \min\limits_{t\geq 0} \{D[{\bf p}(t)] =  \mu\,\nu\}.
\end{equation}

\begin{remark}
Thanks to the previous remark, $\mathbb{E}|X_t| \xrightarrow{t \to \infty} +\infty$. By the identity $|X_t| = X^{+}_t + X^{-}_t$ and the obvious symmetry, we deduce that $D(t) = \mathbb{E}[X^{-}_t]$ is also unbounded as $t \to \infty$. This observation ensures the finiteness of the $t_*$.
\end{remark}

At the level of the agent-based system, after the first time when there is no cash in the bank, i.e., when $t \geq t^{\textrm{stoc}}_* := \min\{\tau > 0 \mid B_c(\tau) = 0\}$, the analysis is much more involved so heuristic reasoning plays a major role in this manuscript. We notice that we have the following basic relations for all time $t \geq 0$:
\begin{equation}\label{eq:evolution_Bc}
B_c = B_* - \sum_{i=1}^N S^{-}_i,~~ \dd B_c = -\sum_{i=1}^N \dd S^{-}_i,
\end{equation}
and $B_c \geq 0$. Therefore, the evolution of $B_c$ is much faster than the evolution of each of the $S_i$'s, indicating that \eqref{eq:evolution_Bc} is really a ``fast'' dynamics compared to \eqref{eq:SDE}. These observations motivate the next definition:

\begin{definition}[\textbf{Mean-field equation --- Phase \RN{2}}]
\label{def:phase2}
We define $\widetilde{S}_1$ (for $t \geq t^{\textrm{stoc}}_*$) to be the compound Poisson process satisfying the following SDE:
\begin{equation}
\label{eq:SDE_phase2_limit}
\dd \widetilde{S}_1 = -\left(1 - \mathbbm{1}_{(-\infty,0]}(\widetilde{S}_1)\cdot Z\right)\dd \widetilde{\mathrm{\bf N}}^1_t + \left(1 - (1-Y)\cdot Z\right)\dd \widetilde{\mathrm{\bf M}}^1_t,
\end{equation}
in which $\widetilde{\mathrm{\bf N}}^1_t$ and $\widetilde{\mathrm{\bf M}}^1_t$ are independent Poisson processes running at the unit intensity. Moreover, $Y = Y(t) \sim \mathcal{B}\left(r(t)\right)$ and $Z = Z(t) \sim \mathcal{B}\left(q_0(t)\right)$ are independent Bernoulli random variables with parameters
\begin{equation}\label{eq:r}
  r := \mathbb{P}\left[\widetilde{S}_1 \geq 1\right]
\end{equation}
and \begin{equation}\label{eq:q0}
  q_0 := \frac{\mathbb{P}\left[\widetilde{S}_1=0\right]}{\mathbb{P}\left[\widetilde{S}_1 \leq 0\right]\cdot \mathbb{P}\left[\widetilde{S}_1 \geq 0\right]}.
\end{equation}
respectively.
\end{definition}

We again denote by ${\bf p}(t)=\left(\ldots,p_{-n}(t),\ldots,p_{-1}(t),p_0(t),p_1(t),\ldots,p_n(t),\ldots\right)$ the law of the process $\widetilde{S}_1(t)$ for $t \geq t_*$, i.e. $p_n(t) = \mathbb P\left[\widetilde{S}_1(t) = n\right]$ for $t \geq t_*$. Its time evolution is given by the following dynamics:
\begin{equation}
  \label{eq:Q_2_a}
\frac{\dd}{\dd t} {\bf p}(t) = Q_2[{\bf p}(t)]
\end{equation}
with
\begin{equation}
  \label{eq:Q_2_b}
Q_2[{\bf p}]_n:= \left\{
    \begin{array}{ll}
      \tfrac{rd}{(r+p_0)(d+p_0)}\,p_{n+1} + \tfrac{r}{r+p_0}\,p_{n-1} - \left(\tfrac{rd}{(r+p_0)(d+p_0)} + \tfrac{r}{r+p_0}\right)\,p_n, &\quad  n \leq -1,\\
      p_1 + \tfrac{r}{r+p_0}\,p_{-1} - \left(\tfrac{rd}{(r+p_0)(d+p_0)} + \tfrac{r}{r+p_0}\right)\,p_0, & \quad n= 0, \\
      p_{n+1} + \tfrac{r}{r+p_0}\,p_{n-1}- \left(1+\tfrac{r}{r+p_0}\right)\,p_n, & \quad n \geq 1,
    \end{array}
  \right.
\end{equation}
where
\begin{equation}
  \label{eq:Q_2_c}
  r := \sum_{n\geq 1} p_n \quad \text{and} \quad d := \sum_{n\leq -1} p_n
\end{equation}
represent the proportion of ``rich'' and ``debt'' agents, respectively.

\begin{remark}
We will illustrate the basic intuition behind Definition \ref{def:phase2}. Under the large population limit $N \to \infty$, we denote by ${\bf q}(t) = (q_0(t),q_1(t),\ldots)$ the law of $B_c(t)$. It is easily seen that the transition rates of $B_c$ can be described by figure \ref{fig:evolution_bank} below.

\begin{figure}[!htb]
\centering
\includegraphics[width=.7\textwidth]{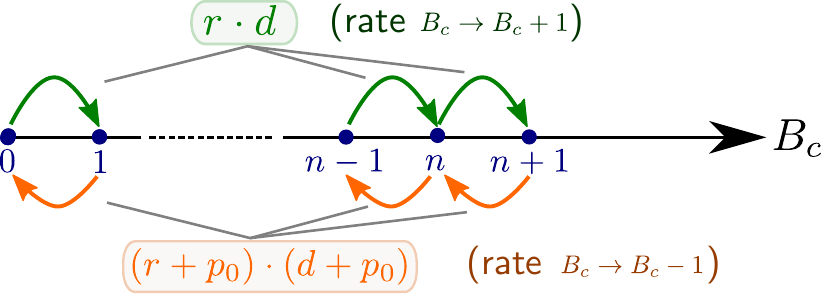}
\caption{Schematic illustration of the evolution of $B_c$ (amount of cash in the bank).}
\label{fig:evolution_bank}
\end{figure}

That is, the transition $B_c \to B_c + 1$ occurs when a ``rich'' agent is being picked ($S_i≥1$) and gives a dollar to an agent in  ``debt'' ($S_j<0$). Similarly, the transition $B_c \to B_c - 1$ happens when an agent without dollar is being picked ($S_i≤0$) and gives a dollar to an agent with debt ($S_j≥0$) and moreover $B_c \geq 1$. As the evolution of $B_c$ is a ``fast'' dynamics, we assume its distribution will relax to its ergodic invariant distribution within a time-scale that is negligible compared to the evolution of each of the $S_i$'s. Thus, from the following detailed balance equation at equilibrium
\[q_n\cdot r\cdot d = q_{n+1}\cdot (r+p_0)\cdot (d + p_0),\] one arrives at $q_n = \left(\frac{r\,d}{(r+p_0)\,(d+p_0)}\right)^n q_0$ for all $n\geq 0$. Taking into account that $\sum_{n \geq 0} q_n = 1$, we end up with
\[q_0 = \frac{p_0}{(r+p_0)\,(d+p_0)},\] which is exactly \eqref{eq:q0}. 
On the other hand, the transition rates of $\widetilde{S}_1$ can be described by figure \ref{fig:evolution_tildeS} below.

\begin{figure}[!htb]
  \centering
  \includegraphics[width=.97\textwidth]{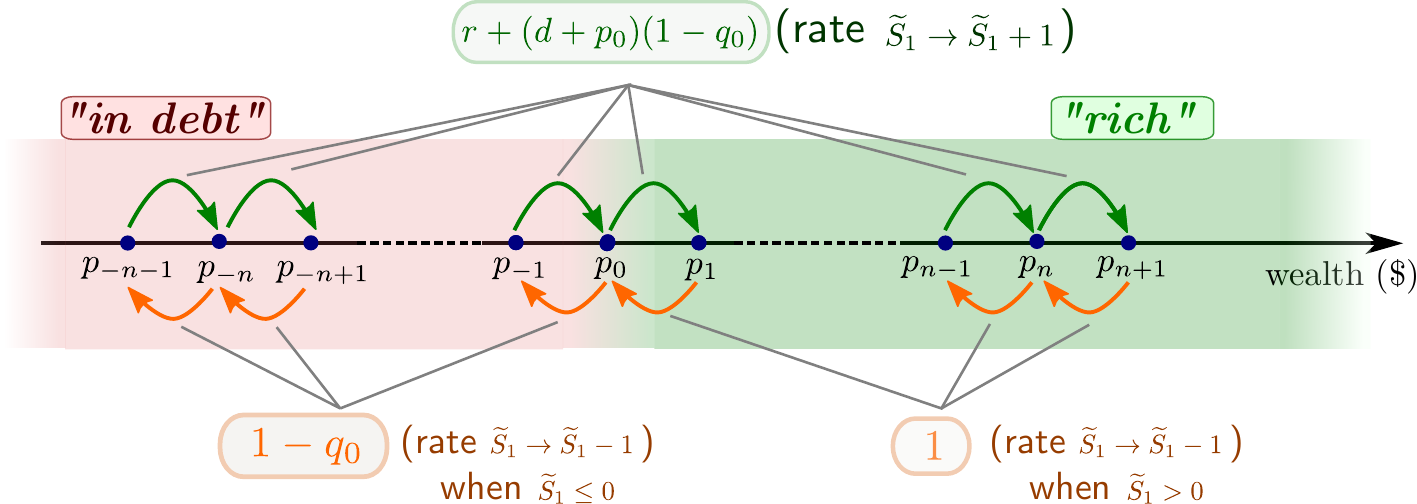}
\caption{Schematic illustration of the evolution of $\widetilde{S}_1$ in \text{\bf Phase II}.}
\label{fig:evolution_tildeS}
\end{figure}

This suggests that the evolution of ${\bf p}(t)$ should obey
\begin{equation}\label{eq:Qtilde}
\frac{\dd}{\dd t} {\bf p}(t) = \widetilde{Q}[{\bf p}(t)],
\end{equation}
with
\begin{equation}\label{eq:Qtilde_expression}
 \widetilde{Q}[{\bf p}]_n= \left\{
    \begin{array}{cclll}
      (1-q_0)(p_{n+1}-p_n)  &-& γ(p_{n-1}-p_n), & \quad n \leq -1,\\
      p_1-p_0  &-& γ(p_{-1}-p_0), & \quad n= 0, \\
      p_{n+1}-p_n  &-& γ(p_{n-1}-p_n), & \quad n \geq 1.
    \end{array}
  \right.
\end{equation}
with $γ=r + (d+p_0)(1-q_0)$. If we insert $q_0 = \frac{p_0}{(r+p_0)\,(d+p_0)}$, then the system \eqref{eq:Qtilde}-\eqref{eq:Qtilde_expression} coincides with \eqref{eq:Q_2_a}-\eqref{eq:Q_2_b}.
\end{remark}

\section{Large time behavior}\label{sec:large_time_limit}
\setcounter{equation}{0}

Although our mean-field analysis is not completely rigorous, we will soon see that our system of ODEs (i.e., \eqref{eq:Q_1_a}-\eqref{eq:Q_1_b} for $t \leq t_*$ and \eqref{eq:Q_2_a}-\eqref{eq:Q_2_b} for $t \geq t_*$) converges to a double-geometric distribution as $t \to \infty$, which resembles an asymmetric Laplace distribution when $\mu$ becomes large. Taking into account the conjecture in \cite{xi_required_2005} and the work of \cite{lanchier_rigorous_2018-1}, as well as several numerical experiments, we believe that the two phase dynamics \eqref{eq:evolution_p} accurately captures the mean-field behavior of the stochastic $N$-agents dynamics as $N \to \infty$.

\subsection{Elementary properties of the limit equation}\label{subsec:3.1}

In \text{\bf Phase II}, the collective debt $d$ \eqref{eq:Q_2_c} is preserved, therefore we can further restrict the affine space where the solution ${\bf p}(t)$ of \eqref{eq:Q_2_a}-\eqref{eq:Q_2_b} lives:
\begin{equation}\label{eq:SUV}
\mathcal{S}_{\mu,\nu} = \{{\bf p} \mid \sum_{n \in \mathbb Z} p_n =1,~p_n \geq 0,~\sum_{n \in \mathbb Z} n\,p_n =\mu,~-\sum_{n \leq 0} n\,p_n = \mu\,\nu\},
\end{equation}
it is straightforward to check that if ${\bf p}(t_*) \in \mathcal{S}_{\mu,\nu}$ then ${\bf p}(t) \in \mathcal{S}_{\mu,\nu}$ for all $t \geq t_*$. Now, we identify the unique equilibrium solution associated with \eqref{eq:Q_2_a}-\eqref{eq:Q_2_b} in this space.

\begin{proposition}\label{prop:equilibrium}
The unique equilibrium solution of \eqref{eq:Q_2_a}-\eqref{eq:Q_2_b} in $\mathcal{S}_{\mu,\nu}$, denoted by ${\bf p}^* = \{p^*_n\}_{n \in \mathbb Z}$, is given by
\begin{equation}\label{eq:doubl-geometric}
p^*_n = \left(\frac{r^*}{r^* + p^*_0}\right)^n p^*_0,~~n \geq 0, \quad p^*_n = \left(\frac{d^*}{d^* + p^*_0}\right)^{-n} p^*_0,~~n \leq 0,
\end{equation}
in which \begin{equation}\label{eq:p*0}
p^*_0 = \left\{
    \begin{array}{ll}
      1/(4\,\nu + 2), & \quad \text{if}~ \mu = 1,\\
      \frac{\mu(\nu+ \frac 12) \,-\, \sqrt{\mu^2\nu^2 + \mu^2\nu + \frac 14}}{(\mu^2-1)/2}, & \quad \text{if}~ \mu \neq 1,
    \end{array}
  \right.
\end{equation}
\begin{equation}\label{eq:r*d*}
r^* = \frac{(\mu-1)\,p^*_0 + 1}{2},~~\textrm{and}~~ d^* = 1 - p^*_0 - r^*.
\end{equation}
\end{proposition}
\begin{proof}
From the evolution equation in \text{\bf Phase II} \eqref{eq:Q_2_a}-\eqref{eq:Q_2_b}, it is straightforward to check that the equilibrium distribution takes the ascertained form \eqref{eq:doubl-geometric}. As we also have ${\bf p}^* \in \mathcal{S}_{\mu,\nu}$, we impose that
\begin{equation*}
\sum_{n\in \mathbb Z} n\,p^*_n = \mu,~~\text{and}~~ -\sum_{n \leq 0} n\,p^*_n = \mu\,\nu.
\end{equation*}
These constraints lead us to
\begin{equation*}
r^*\,(r^* + p^*_0) = p^*_0\,\left(\mu\,\nu + \mu\right),~~\text{and}~~ d^*\,(d^* + p^*_0) = p^*_0\,\mu\,\nu.
\end{equation*}
Combining these relations with the elementary identity $r^* + p^*_0 + d^* = 1$, we obtain the desired result.
\end{proof}

\begin{remark}
Under the settings of Proposition \ref{prop:equilibrium}, if we take $\mu \gg 1$ while keeping $\nu$ fixed, the two-sided geometric distribution ${\bf p}^*$ \eqref{eq:doubl-geometric} can be well-approximated by the asymmetric Laplace distribution given by \eqref{eq:Laplace} and conjectured in \cite{lanchier_rigorous_2018-1,xi_required_2005}. Furthermore, we have an estimation of the parameters $ρ_0, \,α, \,β$ of this distribution with
 \begin{equation}\label{eq:parameters}
   \rho_0 \sim \frac{1}{\mu}\left(\sqrt{1+\nu}-\sqrt{\nu}\right)^2,~~ \alpha \sim \frac{1}{\mu}\left(1 - \sqrt{\frac{\nu}{1+\nu}}\right), ~~ \beta \sim \frac{1}{\mu}\left(\sqrt{\frac{1+\nu}{\nu}} - 1\right).
\end{equation}
Indeed, by taking $\mu \gg 1$ in \eqref{eq:p*0}, we can perform the following simple asymptotic analysis for $p^*_0$:
\begin{align*}
p^*_0 &= \frac{\mu(\nu+ \frac 12) \,-\, \sqrt{\mu^2\nu^2 + \mu^2\nu + \frac 14}}{(\mu^2-1)/2} = \frac{1}{2\left(\mu(\nu + \frac 12) + \sqrt{\mu^2\nu^2 + \mu^2\nu + \frac 14} \right)} \\
&\sim \frac{1}{\mu}\cdot \frac{1}{1+2\,\nu + 2\,\sqrt{\nu^2 + \nu}} = \frac{1}{\mu}\cdot \left(\frac{1}{\sqrt{1 + \nu} + \sqrt{\nu}}\right)^2 = \frac{1}{\mu}\cdot \left(\sqrt{1+\nu}-\sqrt{\nu}\right)^2.
\end{align*}
Similarly, the asymptotic behavior of $\alpha$ can be seem from the asymptotic analysis for $\log\left(1 \big\slash\frac{r^*}{r^* + p^*_0}\right)$. Thanks to \eqref{eq:p*0}, \eqref{eq:r*d*} and the asymptotic argument for $p^*_0$ above, we have
\begin{align*}
\log\left(1 \big\slash\frac{r^*}{r^* + p^*_0}\right) &= \log\left(\frac{(\mu +1)\,p^*_0 + 1}{(\mu -1)\,p^*_0 + 1}\right) \sim \log\left(\frac{\frac{(\mu +1)}{\mu}\left(\sqrt{1+\nu}-\sqrt{\mu}\right)^2 +1}{\frac{(\mu -1)}{\mu}\left(\sqrt{1+\nu}-\sqrt{\mu}\right)^2 +1}
\right) \\
&= \log\left(\frac{1 + \left(\sqrt{1+\nu}-\sqrt{\nu}\right)^2 + \frac{\left(\sqrt{1+\nu}-\sqrt{\nu}\right)^2}{\mu}}{1 + \left(\sqrt{1+\nu}-\sqrt{\nu}\right)^2 - \frac{\left(\sqrt{1+\nu}-\sqrt{\nu}\right)^2}{\mu}}\right) \\
&\sim 2\,\frac{\frac{\left(\sqrt{1+\nu}-\sqrt{\nu}\right)^2}{\mu}}{1 + \left(\sqrt{1+\nu}-\sqrt{\nu}\right)^2} \\
&= \frac{1}{\mu}\cdot \frac{\left(\sqrt{1+\nu}-\sqrt{\nu}\right)^2}{1 + \nu - \sqrt{\nu\,(1+\nu)}} = \frac{1}{\mu}\cdot\left(1 - \sqrt{\frac{\nu}{1+\nu}}\right).
\end{align*}
The argument for $\beta$ is pretty similar and hence will be omitted here.
\end{remark}

\subsection{Convergence to asymmetric Laplace distribution}\label{subsec:3.2}

The main goal of this section is to prove the convergence of the solution ${\bf p}(t)$ in \text{\bf Phase II} \eqref{eq:Q_2_a}-\eqref{eq:Q_2_b} to its two-sided geometric equilibrium solution ${\bf p}^*$ as $t \to \infty$. The essence of the method consists in studying the (relative) entropy dissipation between the two distributions. We remind the formula for estimating the relative entropy:
\begin{equation}
  \label{eq:KL}
  \DD_{\KL}\left({\bf p}~||~ {\bf q}\right) = \sum\limits_{n\in \mathbb Z} p_n\,\log \left(\frac{p_n}{q_n}\right).
\end{equation}

\begin{theorem}\label{thm:convergence}
Let ${\bf p}(t) = \{p_n(t)\}_{t \geq t_*}$ be the unique solution to \eqref{eq:Q_2_a}-\eqref{eq:Q_2_b} with ${\bf p}(t_*) \in \mathcal{S}_{\mu,\nu}$, then for all $t \geq t_*$ we have
\begin{equation}\label{eq:enropy_dissipation}
\begin{aligned}
\frac{\dd}{\dd t}\, \DD_{\KL}\left({\bf p}(t)~||~ {\bf p}^*\right) &= -\sum\limits_{n\geq 0}r\,\left(\frac{p_{n+1}}{r} - \frac{p_n}{r+p_0}\right)\log\frac{p_{n+1}/r}{p_n/(r+p_0)} \\
&\quad - \sum\limits_{n\leq -1} \frac{r\,d}{r+p_0}\,\left(\frac{p_{n+1}}{d+p_0} - \frac{p_n}{d}\right)\log\frac{p_{n+1}/(d+p_0)}{p_n/d} \leq 0.
\end{aligned}
\end{equation}
Consequently, if $\sum_{n \geq 0} p_n(t_*)\,c^n + \sum_{n \leq 0} p_n(t_*)\,c^{-n} < \infty$ for some $c >1$, then the solution ${\bf p}(t)$ converges strongly in $\ell^p$ for $1 < p < \infty$ as $t \to \infty$ to the two-sided geometric distribution ${\bf p}^*$.
\end{theorem}

\begin{proof}
We split the computation of \[\frac{\dd}{\dd t}\, \DD_{\KL}\left({\bf p}(t)~||~ {\bf p}^*\right) = \frac{\dd}{\dd t}\, \sum\limits_{n\in \mathbb Z} p_n\,\log p_n -  \frac{\dd}{\dd t}\, \sum\limits_{n\in \mathbb Z} p_n\,\log p^*_n \] into two parts. On the one hand, we have
\begin{equation*}
\frac{\dd}{\dd t} \sum\limits_{n\in \mathbb Z} p_n\,\log p^*_n = \sum_{n\geq 1} p'_n\,\log p^*_n + p'_0\,\log p^*_0 + \sum_{n\leq -1} p'_n\,\log p^*_n := \RN{1} + \RN{2} + \RN{3},
\end{equation*}
where \begin{equation}\label{RN1}
\begin{aligned}
\RN{1} &= \sum_{n\geq 1} p_{n+1}\,\left(\log p^*_0 + n\,\log\frac{r^*}{r^* + p^*_0}\right) + \sum_{n\geq 1} \frac{r}{r+p_0}\,p_{n-1}\,\left(\log p^*_0 + n\,\log\frac{r^*}{r^* + p^*_0}\right) \\
&\qquad - \sum_{n\geq 1} \left(1+\frac{r}{r+p_0}\right)\,p_n\,\left(\log p^*_0 + n\,\log\frac{r^*}{r^* + p^*_0}\right),
\end{aligned}
\end{equation}
\begin{equation}\label{RN2}
\RN{2} = \left(p_1 + \frac{r}{r + p_0}\,p_{-1} - \left(\frac{r\,d}{(r+p_0)\,(d+p_0)} + \frac{r}{r+p_0}\right)\,p_0 \right)\,\log p^*_0,
\end{equation}
and \begin{equation}\label{RN3}
\begin{aligned}
\RN{3} &= \sum_{n\leq -1} \frac{r\,d}{(r+p_0)\,(d+p_0)}\,p_{n+1}\,\left(\log p^*_0 + n\,\log\frac{r^*}{r^* + p^*_0}\right) \\
&\quad+ \sum_{n\leq -1} \frac{r}{r+p_0}\,p_{n-1}\,\left(\log p^*_0 + n\,\log\frac{r^*}{r^* + p^*_0}\right) \\
&\quad - \sum_{n\geq 1} \left(\frac{r\,d}{(r+p_0)\,(d+p_0)}+\frac{r}{r+p_0}\right)\,p_n\,\left(\log p^*_0 + n\,\log\frac{r^*}{r^* + p^*_0}\right).
\end{aligned}
\end{equation}
Assembling \eqref{RN1}-\eqref{RN3}, we end up with
\begin{equation}\label{eq:cross_entropy}
\frac{\dd}{\dd t} \sum\limits_{n\in \mathbb Z} p_n\,\log p^*_n = 0.
\end{equation}
On the other hand, straightforward computations lead us to
\begin{equation}\label{eq:entropy}
\begin{aligned}
\frac{\dd}{\dd t} \sum\limits_{n\in \mathbb Z} p_n\,\log p_n &= -\sum\limits_{n\geq 0} \left(p_{n+1}-\frac{r}{r+p_0}\,p_n\right)\,\log\frac{p_{n+1}}{\frac{r}{r+p_0}\,p_n} \\
&\quad - \sum\limits_{n\leq 0} \frac{r}{r+p_0}\,\left(\frac{d}{d+p_0}\,p_{n+1} - p_n\right)\,\log\frac{\frac{d}{d+p_0}\,p_{n+1}}{p_n} \\
&= -\sum\limits_{n\geq 0}r\,\left(\frac{p_{n+1}}{r} - \frac{p_n}{r+p_0}\right)\log\frac{p_{n+1}/r}{p_n/(r+p_0)} \\
&\quad - \sum\limits_{n\leq -1} \frac{r\,d}{r+p_0}\,\left(\frac{p_{n+1}}{d+p_0} - \frac{p_n}{d}\right)\log\frac{p_{n+1}/(d+p_0)}{p_n/d}. 
\end{aligned}
\end{equation}
Combining \eqref{eq:cross_entropy} and \eqref{eq:entropy} gives rise to the advertised entropy dissipation result \eqref{eq:enropy_dissipation}. The last statement of Theorem \ref{thm:convergence} can be handled in a pretty similar way as in the case for the basic unbiased exchange model (without the presence of a bank), see for instance \cite{cao_derivation_2021} or \cite{merle_cutoff_2019}.
\end{proof}

To illustrate the convergence of ${\bf p}(t)$ toward the equilibrium ${\bf p}^*$, we run a simulation and plot the evolution of ${\bf p}(t)$ at different times (see figure \ref{fig:cv_equilibrium}-left) as well as the evolution of the relative entropy $\DD_{\KL}\left({\bf p}(t)~||~ {\bf p}^*\right)$ over time (see figure \ref{fig:cv_equilibrium}-right). Notice that the decay of the relative entropy changes abruptly around $t\approx 200$ which corresponds to the transition of the dynamics from phase I to phase II. During the phase II, the decay of the relative entropy is well-approximated by a decay of the form:
\begin{equation}
  \label{eq:decay_D_KL}
  \DD_{\KL}\left({\bf p}(t)~||~ {\bf p}^*\right) \approx c_1 \exp(-c_2 \sqrt{t}),
\end{equation}
where the coefficients $c_1=.674$ and $c_2=.182$ are estimated using mean-square error criteria. We emphasize here that the ansatz for the decay of the relative entropy $\DD_{\KL}\left({\bf p}(t)~||~ {\bf p}^*\right)$ \eqref{eq:decay_D_KL} is inspired from a very recent work on the vanilla unbiased exchange model \cite{cao_interacting_2022}, which corresponds to the special case of the model at hand where $\nu = 0$.

\begin{figure}[ht]
  \centering
  \includegraphics[width=.97\textwidth]{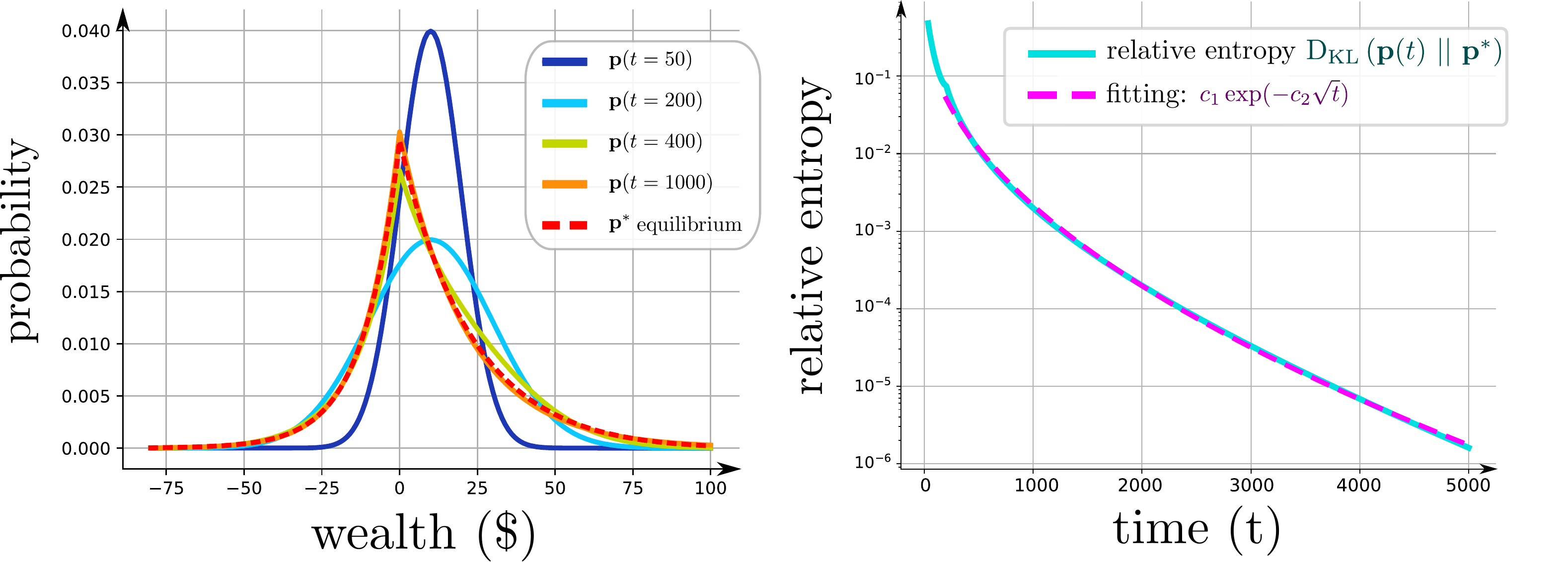}
  \caption{{\bf Left:} evolution of the distribution ${\bf p}(t)$ toward the equilibrium ${\bf p}^*$ \eqref{eq:doubl-geometric}. The dynamics is in phase I till $t\approx 200$. {\bf Right:} evolution of the relative entropy \eqref{eq:KL} between the solution ${\bf p}(t)$ and the equilibrium ${\bf p}^*$. In phase II (i.e. $t>200$), the decay is approximately given by $c_1 \exp(-c_2 \sqrt{t})$ with $c_1=.674$ and $c_2=.182$.}
  \label{fig:cv_equilibrium}
\end{figure}

\subsection{Linearization analysis}\label{subsec:3.3}

We perform a standard linearization analysis near the two-sided geometric equilibrium distribution ${\bf p}^*$. The hope is that for some ranges of parameter choices for $\mu$ and $\nu$, the linearized entropy will decay exponentially fast in time. As we will soon realize, to enlarge the parameter choices of $\mu$ and $\nu$ for which the linearized entropy decreases exponentially fast, one has to find the best possible constants for certain Poincar\'e-type inequalities.

We linearize the ODE system \eqref{eq:Q_2_a}-\eqref{eq:Q_2_b} near its equilibrium ${\bf p}^*$, i.e., we set $p_n = p^*_n + \varepsilon\,w_n$ for $0 < \varepsilon \ll 1$ for all $n \in \mathbb Z$ to obtain the following linearized system around ${\bf p}^*$:

\begin{equation}\label{eq:ODE_linearized}
\frac{\dd}{\dd t} {\bf w} = \mathcal{L}[{\bf w}]
\end{equation}
with
\begin{equation*}
\mathcal{L}[{\bf w}]_n = w_{n+1} - w_n + \frac{r^*}{r^*+p^*_0}\,(w_{n-1}-w_n) + \alpha\,\frac{r^*}{r^*+p^*_0}\,(p^*_{n-1}-p^*_n),\quad \text{for}~ n \geq 1,
\end{equation*}
\begin{equation*}
\begin{aligned}
\mathcal{L}[{\bf w}]_0 &= w_1 + \frac{r^*}{r^*+p^*_0}\,(w_{-1}-w_0) + \alpha\,\frac{r^*}{r^*+p^*_0}\,(p^*_{-1}-p^*_0) \\
&~~- \left(w_0 + (\alpha + \beta)\,p^*_0\right)\,\frac{r^*\,d^*}{(r^*+p^*_0)\,(d^*+p^*_0)},
\end{aligned}
\end{equation*}
and
\begin{equation*}
\begin{aligned}
\mathcal{L}[{\bf w}]_n &= \frac{r^*\,d^*}{(r^*+p^*_0)\,(d^*+p^*_0)}\,(w_{n+1}-w_n) + \frac{r^*}{r^*+p^*_0}\,(w_{n-1}-w_n) \\
&~~+ \alpha\,\frac{r^*}{r^*+p^*_0}\,(p^*_{n-1}-p^*_n) + \frac{r^*\,d^*\,(\alpha + \beta)\,(p^*_{n+1}-p^*_n)}{(r^*+p^*_0)\,(d^*+p^*_0)},\quad \text{for}~ n \leq -1,
\end{aligned}
\end{equation*}
where we set \[\alpha := \frac{\sum\limits_{n \geq 1} w_n}{r^*} - \frac{\sum\limits_{n \geq 0} w_n}{r^*+p^*_0},\quad \text{and}\quad \beta := \frac{\sum\limits_{n \leq -1} w_n}{d^*} - \frac{\sum\limits_{n \leq 0} w_n}{d^*+p^*_0}.\]
It is straightforward to check that \[{\bf w}(t) \in \VV := \{{\bf w} \in \ell^1(\mathbb Z) \mid \sum_{n \in \mathbb Z} w_n =0,~\sum_{n \in \mathbb Z} n\,w_n =0,~\sum_{n \leq 0} n\,w_n =0\}\] for all $t$. Moreover, if we denote the linearized entropy associated with the linearized system \eqref{eq:ODE_linearized} by
\begin{equation}\label{eq:linearized_entropy}
\E[{\bf w}] := \sum\limits_{n \in \mathbb Z} \frac{w^2_n}{p^*_n},
\end{equation}
a direct computation yields
\begin{equation}\label{eq:derivative_of_E}
\begin{aligned}
\frac 12\,\frac{\dd}{\dd t}\, \E[{\bf w}] &= \left(1 - \frac{r^*\,d^*}{(r^*+p^*_0)\,(d^*+p^*_0)}\right)\frac{w^2_0}{p^*_0} + \alpha^2\,r^* + \beta^2\,\frac{r^*\,d^*}{r^*+p^*_0} \\
&~~- \sum\limits_{n\geq 0}\frac{(w_{n+1} - w_n)^2}{p^*_n} - \frac{r^*}{r^*+p^*_0}\,\sum\limits_{n\leq -1}\frac{(w_{n+1} - w_n)^2}{p^*_{n+1}}.
\end{aligned}
\end{equation}
We can bound $\alpha^2$ as follows. Introduce $\{\alpha_n\}_{n \neq 0}$ by defining $\alpha_n = \frac{1}{r^*}$ for $n \geq 1$ and $\alpha_n = \frac{1}{r^*+p^*_0}$ for $n\leq -1$. Similarly, let $\{\lambda_n\}_{n \neq 0}$ be defined by $\lambda_n = \gamma_1$ for $n \geq 1$ and $\lambda_n = \gamma_2$ for $n\leq -1$, where $\gamma_1$ and $\gamma_2$ are constants whose choices will be optimized later. The definition of $\alpha$, together with the fact that ${\bf w} \in \VV$, allows us to deduce that
\begin{equation*}
\begin{aligned}
\alpha^2 &= \left(\sum\limits_{n \neq 0} (\lambda_n\,n + 1)\,\alpha_n\,w_n \right)^2 \\
&\leq \left(\frac{\sum\limits_{n\geq 1} (\gamma_1\,n + 1)^2\,p^*_n}{(r^*)^2} + \frac{\sum\limits_{n\leq -1} (\gamma_2\,n + 1)^2\,p^*_n}{(r^*+p^*_0)^2}\right)\,\sum\limits_{n\neq 0} \frac{w^2_n}{p^*_n}.
\end{aligned}
\end{equation*}
Optimizing the values of $\gamma_1$ and $\gamma_2$ gives rise to $\gamma_1 = -\frac{p^*_0}{2\,r^*+p^*_0}$ and $\gamma_2 = \frac{p^*_0}{2\,d^*+p^*_0}$. Thus, we obtain
\begin{equation}\label{eq:bound_alpha}
\alpha^2 \leq \left(\frac{1}{2\,r^*+p^*_0} + \frac{(d^*)^2}{(2\,d^*+p^*_0)\,(r^*+p^*_0)^2}\right)\,\sum\limits_{n\neq 0} \frac{w^2_n}{p^*_n}.
\end{equation}
In a similar fashion, we also get
\begin{equation}\label{eq:bound_beta}
\beta^2 \leq \left(\frac{1}{2\,d^*+p^*_0} + \frac{(r^*)^2}{(2\,r^*+p^*_0)\,(d^*+p^*_0)^2}\right)\,\sum\limits_{n\neq 0} \frac{w^2_n}{p^*_n},
\end{equation}
as well as
\begin{equation}\label{eq:bound_w_and_p0}
\frac{w^2_0}{p^*_0} \leq \left(\frac{\frac{(r^*)^2}{2\,r^*+p^*_0} + \frac{(d^*)^2}{(2\,d^*+p^*_0)}}{p^*_0 + \frac{(r^*)^2}{2\,r^*+p^*_0} + \frac{(d^*)^2}{(2\,d^*+p^*_0)}}\right)\,\sum\limits_{n \in \mathbb Z} \frac{w^2_n}{p^*_n}.
\end{equation}
Thus, we have the following upper bound:
\begin{equation}\label{eq:upper_bound}
\begin{aligned}
&\left(1 - \frac{r^*\,d^*}{(r^*+p^*_0)\,(d^*+p^*_0)}\right)\frac{w^2_0}{p^*_0} + \alpha^2\,r^* + \beta^2\,\frac{r^*\,d^*}{r^*+p^*_0} \\
&\quad\leq \Bigg(1 - \frac{r^*\,d^*}{(r^*+p^*_0)\,(d^*+p^*_0)} - \frac{r^*}{2\,r^*+p^*_0} - \frac{r^*\,(d^*)^2}{(2\,d^*+p^*_0)\,(r^*+p^*_0)^2} \\
&\quad~- \frac{r^*\,d^*}{(r^*+p^*_0)\,(2\,d^*+p^*_0)} - \frac{(r^*)^3\,d^*}{(r^*+p^*_0)\,(2\,r^*+p^*_0)\,(d^*+p^*_0)^2} \Bigg)\,\frac{w^2_0}{p^*_0} \\
&\qquad~ + \left(\frac{r^*}{2\,r^*+p^*_0} + \frac{r^*\,(d^*)^2}{(2\,d^*+p^*_0)\,(r^*+p^*_0)^2}\right)\,\E[{\bf w}] \\
&\qquad~ + \left(\frac{r^*\,d^*}{(r^*+p^*_0)\,(2\,d^*+p^*_0)} + \frac{(r^*)^3\,d^*}{(r^*+p^*_0)\,(2\,r^*+p^*_0)\,(d^*+p^*_0)^2}\right)\,\E[{\bf w}].
\end{aligned}
\end{equation}
On the other hand, we have the following simple lower bound on $\sum\limits_{n\geq 0}\frac{(w_{n+1} - w_n)^2}{p^*_n}$:
\begin{equation}\label{eq:lower_bound1}
\sum\limits_{n\geq 0}\frac{(w_{n+1} - w_n)^2}{p^*_n} \geq \left(\frac{\sqrt{r^*+p^*_0}-\sqrt{r^*}}{\sqrt{r^*+p^*_0}}\right)^2\,\sum\limits_{n \geq 0} \frac{w^2_n}{p^*_n}.
\end{equation}
Indeed, we have
\begin{align*}
\sum\limits_{n\geq 0}\frac{w^2_n}{p^*_n} &= \sum\limits_{n\geq 0} \frac{1}{p^*_n}\left(\sum\limits_{k \geq n+1}(w_{k-1}-w_k)\right)^2\\
&= \sum\limits_{n\geq 0}\frac{1}{p^*_n}\,\sum\limits_{k\geq n+1} (w_{k-1}-w_k)\,\sum\limits_{\ell \geq n+1} (w_{\ell-1}-w_\ell) \\
&= \sum\limits_{k,\ell\geq 1}(w_{k-1}-w_k)\,(w_{\ell-1}-w_\ell)\,\sum\limits_{0\leq n\leq \min(k-1,\ell-1)} \frac{1}{p^*_n} \\
&= \frac{2\,r^*}{p^*_0}\,\sum\limits_{1\leq k <\ell}(w_{k-1}-w_k)\,(w_{\ell-1}-w_\ell)\,\left(\frac{1}{p^*_k}-\frac{1}{p^*_0}\right) \\
&\qquad+ \frac{r^*}{p^*_0}\,\sum\limits_{k\geq 1} (w_{k-1}-w_k)^2\,\left(\frac{1}{p^*_k}-\frac{1}{p^*_0}\right) \\
&\leq \frac{2\,r^*}{p^*_0}\,\left(\sum\limits_{k\geq 1}\frac{(w_{k-1} - w_k)^2}{p^*_k}\right)^{\frac 12}\left(\sum\limits_{k \geq 1} \frac{w^2_n}{p^*_n}\right)^{\frac 12} + \frac{r^*}{p^*_0}\,\sum\limits_{k\geq 1}\frac{(w_{k-1} - w_k)^2}{p^*_k} \\
&= \frac{2\,\sqrt{r^*\,(r^*+p^*_0)}}{p^*_0}\,\left(\sum\limits_{k\geq 1}\frac{(w_{k-1} - w_k)^2}{p^*_{k-1}}\right)^{\frac 12}\left(\sum\limits_{k \geq 0} \frac{w^2_n}{p^*_n}\right)^{\frac 12} \\
&\qquad + \frac{r^*+p^*_0}{p^*_0}\,\sum\limits_{k\geq 1}\frac{(w_{k-1} - w_k)^2}{p^*_{k-1}},
\end{align*}
from which the advertised lower bound \eqref{eq:lower_bound1} follows. By a parallel reasoning, we also obtain
\begin{equation}\label{eq:lower_bound2}
\sum\limits_{n\leq -1}\frac{(w_{n+1} - w_n)^2}{p^*_{n+1}} \geq \left(\frac{\sqrt{d^*+p^*_0}-\sqrt{d^*}}{\sqrt{d^*+p^*_0}}\right)^2\,\sum\limits_{n \leq 0} \frac{w^2_n}{p^*_n}.
\end{equation}
Combining \eqref{eq:lower_bound1} and \eqref{eq:lower_bound2} leads us to the following lower bound:
\begin{equation}\label{eq:lower_bound}
\begin{aligned}
&\sum\limits_{n\geq 0}\frac{(w_{n+1} - w_n)^2}{p^*_n} + \frac{r^*}{r^*+p^*_0}\,\sum\limits_{n\leq -1}\frac{(w_{n+1} - w_n)^2}{p^*_{n+1}} \\
&\geq \left(\frac{\sqrt{r^*+p^*_0}-\sqrt{r^*}}{\sqrt{r^*+p^*_0}}\right)^2\,\sum\limits_{n \geq 0} \frac{w^2_n}{p^*_n} + \frac{r^*}{r^*+p^*_0}\,\left(\frac{\sqrt{d^*+p^*_0}-\sqrt{d^*}}{\sqrt{d^*+p^*_0}}\right)^2\,\sum\limits_{n \leq 0} \frac{w^2_n}{p^*_n} \\
&\geq \min\left\{\left(\frac{\sqrt{r^*+p^*_0}-\sqrt{r^*}}{\sqrt{r^*+p^*_0}}\right)^2,\,
\frac{r^*}{r^*+p^*_0}\,\left(\frac{\sqrt{d^*+p^*_0}-\sqrt{d^*}}{\sqrt{d^*+p^*_0}}\right)^2 \right\}\,\E[{\bf w}] \\
&~~ + \max\left\{\left(\frac{\sqrt{r^*+p^*_0}-\sqrt{r^*}}{\sqrt{r^*+p^*_0}}\right)^2,\,
\frac{r^*}{r^*+p^*_0}\,\left(\frac{\sqrt{d^*+p^*_0}-\sqrt{d^*}}{\sqrt{d^*+p^*_0}}\right)^2 \right\}\,\frac{w^2_0}{p^*_0}.
\end{aligned}
\end{equation}
To alleviate the writing, we define
\begin{equation}\label{def:C1}
\begin{aligned}
C_1 &= 1 - \frac{r^*\,d^*}{(r^*+p^*_0)\,(d^*+p^*_0)} - \frac{r^*}{2\,r^*+p^*_0} - \frac{r^*\,(d^*)^2}{(2\,d^*+p^*_0)\,(r^*+p^*_0)^2}\\
&\quad- \frac{r^*\,d^*}{(r^*+p^*_0)\,(2\,d^*+p^*_0)} - \frac{(r^*)^3\,d^*}{(r^*+p^*_0)\,(2\,r^*+p^*_0)\,(d^*+p^*_0)^2},
\end{aligned}
\end{equation}
\begin{equation}\label{def:C2}
C_2 = \max\left\{\left(\frac{\sqrt{r^*+p^*_0}-\sqrt{r^*}}{\sqrt{r^*+p^*_0}}\right)^2,\,
\frac{r^*}{r^*+p^*_0}\,\left(\frac{\sqrt{d^*+p^*_0}-\sqrt{d^*}}{\sqrt{d^*+p^*_0}}\right)^2 \right\},
\end{equation}
\begin{equation}\label{def:C3}
\begin{aligned}
C_3 &= \frac{r^*}{2\,r^*+p^*_0} + \frac{r^*\,(d^*)^2}{(2\,d^*+p^*_0)\,(r^*+p^*_0)^2} + \frac{r^*\,d^*}{(r^*+p^*_0)\,(2\,d^*+p^*_0)} \\
&\qquad  + \frac{(r^*)^3\,d^*}{(r^*+p^*_0)\,(2\,r^*+p^*_0)\,(d^*+p^*_0)^2},
\end{aligned}
\end{equation}
\begin{equation}\label{def:C4}
C_4 = \min\left\{\left(\frac{\sqrt{r^*+p^*_0}-\sqrt{r^*}}{\sqrt{r^*+p^*_0}}\right)^2,\,
\frac{r^*}{r^*+p^*_0}\,\left(\frac{\sqrt{d^*+p^*_0}-\sqrt{d^*}}{\sqrt{d^*+p^*_0}}\right)^2 \right\},
\end{equation}
and \begin{equation}\label{def:gamma}
\gamma = \frac{\frac{(r^*)^2}{2\,r^*+p^*_0} + \frac{(d^*)^2}{(2\,d^*+p^*_0)}}{p^*_0 + \frac{(r^*)^2}{2\,r^*+p^*_0} + \frac{(d^*)^2}{(2\,d^*+p^*_0)}}.
\end{equation}
Then we arrive at the following differential inequality for the linearized entropy $\E[{\bf w}]$:
\begin{equation}\label{eq:inequality}
\frac 12\,\frac{\dd}{\dd t}\, \E \leq -(C_2 - C_1)\,\frac{w^2_0}{p^*_0} - (C_4 - C_3)\,\E.
\end{equation}
If we recall the upper bound \eqref{eq:bound_w_and_p0} on $\frac{w^2_0}{p^*_0}$, in order to have an exponential decay result for the linearized entropy, it suffices to ensure that
\begin{equation}\label{eq:sufficiency_cond}
C_4 - C_3 - \gamma\,(C_1 - C_2)\,\mathbbm{1}\{C_1 > C_2\} > 0.
\end{equation}
For convenience, we denote by $\mathcal{G}_{\mu,\nu}$ the collection of the parameter values $(\mu,\nu) \in \mathbb{R}^2_+$ for which the condition \eqref{eq:sufficiency_cond} is satisfied. We claim that this set is non-empty, whence the linearized entropy will decay exponentially fast in time for all $(\mu,\nu) \in \mathcal{G}_{\mu,\nu}$.

\begin{theorem}\label{thm:exponential_decay_linearized_entropy}
The set $\mathcal{G}_{\mu,\nu}$ is not empty. Consequently, the linearized entropy $\E$ defined by \eqref{eq:linearized_entropy} decreases exponentially fast in time if $(\mu,\nu) \in \mathcal{G}_{\mu,\nu}$.
\end{theorem}

\begin{proof}
To demonstrate the non-emptiness of the set $\mathcal{G}_{\mu,\nu}$ it suffices to present a specific example. For this purpose, we can take $(\mu,\nu) = (0.01,0.001)$ to obtain $C_4 - C_3 - \gamma\,(C_1 - C_2)\,\mathbbm{1}\{C_1 > C_2\} \approx 1.6647\cdot 10^{-5} > 0$. The second part of the proposition follows immediately from the definition of $\mathcal{G}_{\mu,\nu}$ and the differential inequality \eqref{eq:inequality}.
\end{proof}

\begin{remark}
The lower bounds \eqref{eq:lower_bound1} and \eqref{eq:lower_bound2} are far from being optimal as their arguments rely sole on some elementary Cauchy-Schwarz inequalities and the constraint that ${\bf w} \in \VV$  has not been exploited at all. We speculate that these lower bounds (Poincar\'e-type inequalities) can somehow be significantly sharpened, leading to a refined version of Theorem \ref{thm:exponential_decay_linearized_entropy}.
\end{remark}

\section{Debt induced wealth inequality}\label{sec:Gini_comparison}
\setcounter{equation}{0}

We would like to conclude our analysis of the model by investigating the effect of the bank on the inequality of the wealth distribution, i.e. does the bank increase or decrease inequality? To measure inequality, we use the Gini index $G$ which is usually an indicator between $0$ (no inequality) and $1$ (maximum inequality) and measures the inequality in the wealth distribution, and a higher Gini index indicates worse equality in the society (in terms of the distribution of wealth).

\begin{definition}[\textbf{Gini index}]
For a given probability mass function ${\bf p} \in \mathcal{P}(\mathbb Z)$ with mean $\mu \in \mathbb{R}_+$, the Gini index of ${\bf p}$ is defined by
\begin{equation}\label{def1:Gini}
G[{\bf p}] = \frac{1}{2\,\mu} \sum\limits_{i\in \mathbb Z}\sum\limits_{j \in \mathbb Z} |i-j|\,p_i\,p_j.
\end{equation}
Equivalently, a probabilistic definition of $G[{\bf p}]$ is given by
\begin{equation}\label{def2:Gini}
G[{\bf p}] = \frac{1}{2\,\mu} \mathbb{E}[|\overbar{S} - \overbar{S}'|],
\end{equation}
where $\overbar{S}$ has ${\bf p}$ as its probability mass function and $\overbar{S}'$ is an independent copy of $\overbar{S}$.
\end{definition}

\begin{remark}
Without debt, i.e. if $p_n=0$ for $n<0$ or $\overbar{S}≥0$, one can show that the Gini index $G$ is always between $0$ (no inequality) and $1$ (maximum inequality). Indeed, using triangular inequality:
\begin{equation}
  G[{\bf p}] ≤ \frac{1}{2\,\mu} \big(\mathbb{E}[|\overbar{S}|] + \mathbb{E}[|\overbar{S}'|]\big)
  \;= \; \frac{1}{2\,\mu} \big(\mathbb{E}[\,\overbar{S}\,] \,+\, \mathbb{E}[\,\overbar{S}'\,]\big)    \; \;=  \;1.
\end{equation}
However, the Gini index is no longer bounded if the wealth distribution includes debts. For instance, taking a Gaussian distribution with mean $μ>0$ and variance $σ^2$, i.e. $X\sim \mathcal{N}(μ,σ^2)$, we obtain:
\begin{eqnarray}
  G[X] &=& \frac{1}{2μ} \mathbb{E}[|Z|] \qquad \text{with } Z\sim \mathcal{N}(0,2σ^2) \\
  \label{eq:gini_gaussian_distribution}
  &=& \frac{1}{2μ}\frac{2σ}{\sqrt{π}}
\end{eqnarray}
Thus, if $σ \gg μ$, the Gini index could be arbitrary large.
\end{remark}

We would like to compare the evolution of Gini index for the basic unbiased exchange model and the model at hand (with bank and debt) at the level of the deterministic ODE system, starting from the same initial distribution. First, we recall that the unbiased exchange model investigated in \cite{cao_derivation_2021,graham_rate_2009,lanchier_rigorous_2017,merle_cutoff_2019} is a special case of the model studied in this work with $\nu = 0$, meaning that the bank does not exist and agents are not allowed to go into debt. The rigorous mean-field analysis shown in \cite{cao_derivation_2021,graham_rate_2009,merle_cutoff_2019} implies that, if we denote by ${\bf q}(t)=\left(q_0(t),q_1(t),\ldots\right)$ the law of the amount of dollars a typical agent has as $N \to \infty$, its time evolution will be given by:
\begin{equation}
\label{eq:vanilla_unbiased_exchange}
\frac{\dd}{\dd t} {\bf q}(t) = \,Q[{\bf q}(t)]
\end{equation}
with
\begin{equation}
  \label{eq:Q}
   Q[{\bf q}]_n:= \left\{
    \begin{array}{ll}
      q_{n+1} + \bar{r}\,q_{n-1} - (1 + \bar{r})\,q_n, &\quad \text{for}~ n \geq 1,\\
      q_1 - \bar{r}\,q_0 & \quad \text{for}~ n= 0, \\
    \end{array}
  \right.
\end{equation}
where $\bar{r} := \sum_{n\geq 1} q_n$.

Intuitively, the model with bank investigated in this manuscript permits agents without dollars in their pocket or agents in debt to be picked to give, thus we speculate that the Gini index of the distribution ${\bf p}(t)$ solution of \eqref{eq:evolution_p} is always larger than the corresponding Gini index of the distribution ${\bf q}$ (if they start from the same initial condition).

We provide in figure \ref{fig:gini_index}-left the evolution of the Gini index for different values of $ν$ (which measures the ratio of the initial total amount of cash in the bank to the agents' combined initial wealth). Without the bank, i.e., when $ν=0$, the Gini index quickly reaches its maximum around $G\approx .8$. However, when the bank comes into play ($ν≥.25$), the Gini index could eventually exceed $1$. Moreover, the phases I and II are clearly seen in the evolution of the Gini index. In phase I, the Gini index grows like $\mathcal{O}(\sqrt{t})$ which is due to the diffusive nature of the dynamics coupled with formula \eqref{eq:gini_gaussian_distribution}. In phase II, the growth of the Gini index starts to saturate and the curve converges to its maximum. In figure \ref{fig:gini_index}-right, we provide the corresponding wealth distribution at the final time of the computation, i.e. ${\bf p}(t=5000)$, depending on the wealth of the bank $ν$. Without bank ($ν=0$), the distribution is exponential but it becomes an asymmetric two-sided exponential for $ν>0$. Notice that in the case where $ν=5$ (which means that the bank has $5$ times more resources than the agents), the dynamics has not reached yet phase II at time $t=5000$ as we observe in the evolution of the Gini index $G$ in figure \ref{fig:gini_index}-left. As a result, the distribution ${\bf p}(t=5000)$ for $ν=5$ is still far away from equilibrium and the characteristic cusp at zero dollar has not yet appeared.

\begin{figure}[ht]
  \centering
  \includegraphics[width=.97\textwidth]{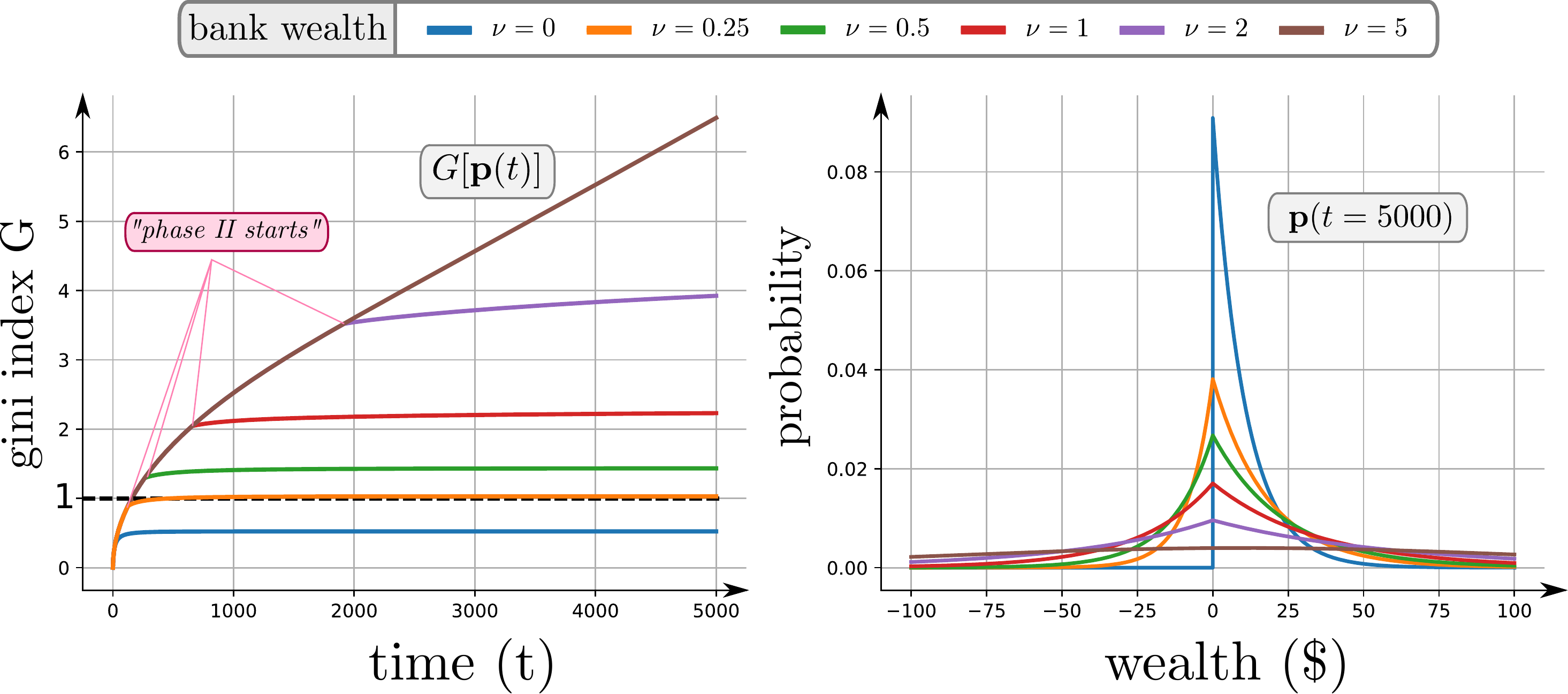}
  \caption{{\bf Left}: evolution of the Gini index $G$ \eqref{def1:Gini} for various resources of the bank $ν$. Notice that the Gini index could exceed $1$ when debt is present. {\bf Right}: the corresponding wealth distribution ${\bf p}$ at the final time of the computation $t=5000$. When $ν=5$, the dynamics has not yet reached the equilibrium distribution.}
  \label{fig:gini_index}
\end{figure}

Although we are not able to provide a complete proof of this natural and compelling conjecture based on heuristic reasoning and numerical simulations. The following simple observation is elementary:

\begin{proposition}\label{ppo:gini_increase}
Assume that ${\bf p}(t)$ is the solution of \eqref{eq:Q_1_a}-\eqref{eq:Q_1_b} for $t \in [0,t_*]$ with ${\bf p}(0) \in \mathcal{S}^+_\mu$, then the Gini index of ${\bf p}(t)$ is non-decreasing with respect to time.
\end{proposition}
\begin{proof}
We resort to the probabilistic representation of the Gini index \eqref{def2:Gini}, so we introduce independent random variables $\overbar{S}$ and $\overbar{S}'$ having ${\bf p}$ as their common probability mass function. If we set $Z = |\overbar{S} - \overbar{S}'|$ and let ${\bf z} = \{z_n\}_{n\geq 0}$ to be law of $Z$, it is trivial to see that the evolution of ${\bf z}$ satisfies
\begin{equation}
  \label{eq:Z}
   \frac{\dd}{\dd t} z_n = \left\{
    \begin{array}{ll}
      2\,\left(z_{n+1} + z_{n-1} - 2\,z_n\right) & \quad \text{for}~ n \geq 1,\\
      2\,\left(z_1 - 2\,z_0\right) & \quad \text{for}~ n= 0, \\
    \end{array}
  \right.
\end{equation}
from which we deduce that
\begin{equation*}
\frac{\dd}{\dd t} G[{\bf p}] = \frac{1}{2\,\mu} \mathbb{E}[Z] = \frac{1}{\mu}\sum\limits_{n \geq 1} n\,(z_{n+1} - 2\,z_n + z_{n-1}) = \frac{z_0}{\mu} \geq 0.
\end{equation*}
In other words, the phase I evolution tends to accentuate the wealth inequality and the proof of Proposition \ref{ppo:gini_increase} is completed.
\end{proof}

\section{Conclusion}

In this manuscript, the so-called unbiased exchange model with collective debt limit is investigated, which can be viewed as an extension of the vanilla unbiased exchange model proposed in \cite{dragulescu_statistical_2000} and revisited in \cite{cao_derivation_2021,lanchier_rigorous_2017,merle_cutoff_2019}. Although the inclusion of bank creates additional difficulty in the analysis and only a formal mean-field argument is presented, we found that the prediction of the asymptotic distribution of wealth as $N \to \infty$ and $t \to \infty$ based on our two-phase dynamics \eqref{eq:evolution_p} agrees with the results reported/conjectured in earlier work \cite{lanchier_rigorous_2018-1,xi_required_2005}. To the best our of knowledge, there are no (even formal) mean-field analysis for the model at hand prior to the present work and we believe that our work leaves many open and challenging questions to be investigated on a more rigorous ground. For instance, is it possible to provide a rigorous proof of the propagation of chaos as $N \to \infty$ to arrive at the two-phase dynamics \eqref{eq:evolution_p}~? How can we extend the recent work \cite{cao_explicit_2021} to justify the ansatz for the decay of the relative entropy $\DD_{\KL}\left({\bf p}(t)~||~ {\bf p}^*\right)$ conjectured in \eqref{eq:decay_D_KL} (for a fairly generic initial datum)~? Is it possible to show rigorously the monotonicity of the Gini index of ${\bf p}(t)$ (solution of \eqref{eq:evolution_p}) with respect to $\nu$ during phase II~? The answer to these questions will enable us to have a better understanding of the role played by the bank and debt.

Lastly, we emphasize that econophysics models involving bank and debt have not received enough attention in the past few years, and the model investigated in this manuscript can be viewed as the representative of a promising and intriguing direction for further research in the area of econophysics. For instance, it would be interesting to introduce a bank in other closely related variants of the basic unbiased exchange model, such as the so-called poor-biased or rich-biased exchange model proposed in \cite{cao_derivation_2021}.

\end{document}